\documentclass[11pt,twoside,reqno]{amsart}
\usepackage{amsmath, amsfonts, amsthm, amssymb,amscd, graphicx, amscd}
\usepackage{float,epsf}
\usepackage[english]{babel}
\usepackage{enumerate}
\usepackage{tikz}
\usepackage{hyperref}
 \usepackage{float} 
\usepackage{srcltx}
\usepackage{subfigure}
\usepackage{graphics}
\usepackage{epsfig}

\usepackage{geometry}
\usepackage{verbatim}
\usepackage{mathrsfs}

\newtheorem{theorem}{Theorem}[section]
\newtheorem{lemma}[theorem]{Lemma}

\newtheorem{proposition}[theorem]{Proposition}
\newtheorem{corollary}[theorem]{Corollary}
\newtheorem{remark}[theorem]{Remark}

\newtheorem{question}[theorem]{Question}

\numberwithin{equation}{section}
\numberwithin{figure}{section}

\makeatletter
\@namedef{subjclassname@2020}{Mathematics Subject Classification 2020}
\makeatother

\def\intave#1{\int_{#1}\hbox{\llap{$\raise2.3pt\hbox{\vrule
height.9pt width7pt}\phantom{\scriptstyle{#1}}\mkern-2mu$}}}

\everymath{\displaystyle}

\title{Concentration breaking on two optimization problems}
\usepackage[misc]{ifsym}

\author{Yong Huang}
\address{School of Mathematics, Hunan University, Changsha, Hunan, China.}
\email{huangyong@hnu.edu.cn}

\author{Qinfeng Li
}
\address{School of Mathematics, Hunan University, Changsha, Hunan, China.}
\email{liqinfeng1989@gmail.com}

\author{Qiuqi Li }
\address{School of Mathematics, Hunan University, Changsha, Hunan, China.}
\email{qli28@hnu.edu.cn}

\thanks{Research of Yong Huang was supported by the National Science Fund (No. 11625103, 12171144), Hunan Science and Technology Planning Project(No. 2019RS3016). Research of Qinfeng Li was supported by the National Science Fund for Youth Scholars (No. 12101215). Research of Qiuqi Li was supported by the National Science Fund for Youth Scholars (No. 12101216 ) and the Natural Science Fund of Hunan Province (No. 2022JJ40030). }


\begin{document}

\begin{abstract}
In the present paper, we study the boundary concentration breaking phenomena on two thermal insulation problems considered on Lipschitz domains, based on Serrin's overdetermined results, perturbation argument and comparison of Laplacian eigenvalues with different boundary conditions. Since neither of the functionals in the two problems is $C^1$, another key ingredient is to obtain the global H\"older regularity of minimizers to both problems on Lipschitz domains. Also, exact dependence on domain of breaking thresholds is also given in the first problem, and the breaking values are obtained in the second problem on ball domains, which are related to $2\pi$ in dimension $2$.

\end{abstract}

\maketitle

\vskip0.2in



\section{Introduction}
The boundary distribution of minimizers is the following two optimization problems from thermal insulation background
\begin{align}
\label{1problem}
    \min\left\{\frac{\int_{\Omega}|\nabla u|^2\,dx+\frac{1}{m}\left(\int_{\partial \Omega}|u|\,d\sigma\right)^2}{\left(\int_{\Omega}u\,dx\right)^2}:u\in H^1(\Omega)\setminus \{0\}\right\},
\end{align}and
\begin{align}
\label{2problem}
    \min\left\{\frac{\int_{\Omega}|\nabla u|^2\,dx+\frac{1}{m}\left(\int_{\partial \Omega}|u|\,d\sigma\right)^2}{\int_{\Omega}u^2\,dx}:u\in H^1(\Omega)\setminus \{0\}\right\},
\end{align}where $\Omega$ represents a thermal body and $m$ represents the total amount of thermal insulation material around $\partial \Omega$. According to \cite{BBN} and \cite{BBN1}, the first problem \eqref{1problem} corresponds to maximizing heat content when the heat source is uniformly distributed, and the second problem is exactly for minimizing decay of temperature when there is no heat source. Let $u_m$ be a solution to \eqref{1problem} or \eqref{2problem}, and define $h^{opt}_m\in L^1(\partial \Omega)$ by
\begin{align}
\label{material}
    h^{opt}_m(\sigma)=\frac{m|u_m(\sigma)|}{\int_{\partial \Omega}|u_m(\sigma)|\,d\sigma},
\end{align}then as shown in \cite{BBN1}, $h^{opt}_m$ gives the optimal distribution of thermal insulation material in both problems. Since $h^{opt}_m$ and $u_m$ are the same on the boundary up to a constant factor, in order to know the distribution of insulation material on the boundary, it suffices to understand the boundary behavior of $u_m$.

Concerning the first problem of maximizing heat content, it is easy to see that if $\Omega$ has sufficient regularity, then up to a constant factor, $u_m$ satisfies the following Euler-Lagrange equation in standard weak sense.
\begin{align}
    \label{EulerLagrangeforuintro}
\begin{cases}
-\Delta u=1 \quad &\mbox{in $\Omega$}\\
\frac{\partial u}{\partial \nu}=-\frac{1}{m}\int_{\partial \Omega}|u| d\sigma &\mbox{on $\partial \Omega \cap \{u\ne 0\}$}\\
\frac{\partial u}{\partial \nu}\in \left[-\frac{1}{m}\int_{\partial \Omega}|u| d\sigma,\frac{1}{m}\int_{\partial \Omega}|u| d\sigma\right] &\mbox{on $\partial \Omega \cap \{u=0\}$}.
\end{cases}    
\end{align}
Given a general domain $\Omega$, $\eqref{EulerLagrangeforuintro}_3$ makes the linear equation difficult to solve, and even numerical method fails since we do not know where $u_m$ vanishes on $\partial \Omega$. There is an interesting example in \cite{Buttazzo}, see also \cite[Example 6.3]{DLW}, saying that if $\Omega$ is an anulus, then there exists $m_1>0$ such that when $m>m_1$, the insulation material covers the whole boundary, while when $m<m_1$, the insulator only concentrates on the inner boundary and leaves the outer boundary unprotected. On the other hand, by \cite{Buttazzo}, see also \cite[Example 2.2]{BBN}, when $\Omega$ is a ball, for any $m>0$, the insulator is always uniformly distributed on $\partial \Omega$, and thus strictly positive. 

The two examples above only involve radially symmetric domains, and it is then natural to ask the following question:
\begin{question}
\label{ques1}
For what kind of shape does the insulator vanish nowhere on the boundary? 
\end{question}
If the question can be answered, then we can classify domains over which the annoying condition $\eqref{EulerLagrangeforuintro}_3$ can no longer be worried about and thus the optimal insulation of material can be figured out on such domains via numerical methods.

In the following, we denote the perimeter of $\Omega$ by $P(\Omega)$, which is equivalent to the $(n-1)$-Hausdorff measure of the reduced boundary of $\Omega$.

We restrict our study to connected domains, and it turns out that whether or not $u_m$ is nowhere vanishing depends not only on the geometry of domain, but also on the given total amount of material. In fact, we prove that a concentration breaking phenomenon happens for any $C^2$ domains which are not balls.
\begin{theorem}
\label{concentration}
Let $\Omega \subset\mathbb{R}^n$ be a connected $C^2$ domain whch is not a ball, and let $u_m$ be a solution to \eqref{1problem}. Then there exists $m_1>0$ which depends only on $\Omega$, such that when $m>m_1$, $|u_m|>0$ on $\partial \Omega$, and when $0<m<m_1$, $u_m$ must vanish on a closed subset of $\partial \Omega$ with positive $\mathscr{H}^{n-1}$ measure.
\end{theorem}
Moreover, through the proof of Theorem \ref{concentration}, the exact dependence of $m_1$ on $\Omega$ is given. Precisely, let $u_0$ be the solution to the following equation \begin{align}
\label{u0}
\begin{cases}
-\Delta u=1 \quad &\mbox{in $\Omega$}\\
\frac{\partial u}{\partial \nu}=-\frac{|\Omega|}{P(\Omega)} &\mbox{on $\partial \Omega$.}
\end{cases} 
\end{align}
Then  $$m_1=m_1(\Omega)=\delta(\Omega)\frac{P^2 (\Omega)}{|\Omega|},$$ where $\delta(\Omega)$ is the mean of $u_0$ over $\partial \Omega$ minus the minimum of $u_0$ on $\partial \Omega$, which is a natural geometric quantity associated to $\Omega$, and is strictly positive as long as $\Omega$ is not a ball, due to Serrin's famous overdetermined result, see \cite{Serrin}. We remark that \eqref{u0} is also exploited in the seminal paper \cite{Cabre} by Cabr\'e to give an alternative proof of classical isoperimetric inequality via ABP estimate, but the equation \eqref{u0} itself still leaves several unsolved problems, such as minimizing $\delta(\Omega)$ over polytopes with fixed number of faces, or stability inequalities regarding $\delta(\Omega)$. In this paper we do not deal with these problems though.

Extending Theorem \ref{concentration} to general Lipschitz domains is very subtle, and the difficulties lie in two aspects. First, the derivation of \eqref{EulerLagrangeforuintro} need to be careful since in \eqref{1problem}, the boundary term $\int_{\partial \Omega}|\cdot|\, d\sigma$ is not $C^1$ and we need to prove at least $C^0$ boundary regularity of minimizers in order for \eqref{EulerLagrangeforuintro} to be valid. Luckily, by regularization argument, we can actually prove global $W^{1,p}$ and H\"older regularity, see Theorem \ref{global} below. Second, Serrin's overdetermined result is proved under the setting that the domain is $C^2$ and the solution is classical. Whether or not the symmetry result is true, under the setting that the domain is merely Lipschitz and the overdtermined system admit a weak solution, remains a challenging open question. We will give some comments in the last section.

\medskip

For the second problem \eqref{2problem} of minimizing temperature decay, Bucur-Buttazzo-Nitsch\cite{BBN} prove the following surprising symmetry breaking result.
\begingroup
\renewcommand{\thetheorem}{\Alph{theorem}}
\begin{theorem}\emph{(\cite[Theorem 3.1]{BBN})}
\label{gongming}
Let $\Omega$ be a ball and $u_m$ be a solution to \eqref{2problem}. Then there exists $m_0>0$ such that when $m>m_0$, $u_m$ is radial, while when $m<m_0$, $u_m$ is not radial.
\end{theorem}
\endgroup
The number $m_0$ is determined as the unique value such that the infimum in \eqref{2problem} is equal to the second eigenvalue of Neuman Laplacian. Note that Theorem \ref{gongming} can be also viewed as a concentration breaking result, since essentially the authors show that when $m<m_0$, $u_m$ must vanish on a subset of $\partial \Omega$ with positive $\mathscr{H}^{n-1}$ measure. Physically this means that even for a round thermal body, when we do not have enough material, then in order for lowest temperature decay, the best insulation should not concentrate on the whole boundary. 


Now there are two immediate questions.
\begin{question}
\label{ques2}
Let $m_0$ be as in Theorem \ref{gongming}. What is the precise value of $m_0$?
\end{question}
\begin{question}
\label{ques3}
Does concentration breaking also occur on any shape of bodies in the second problem?
\end{question}

On Question \ref{ques2}, we derive an exact formula for $m_0$.
\begin{theorem}
\label{exact}
Let $\Omega$ be a ball of radius $R$ in $\mathbb{R}^n$ and $m_0$ be the breaking number as in Theorem \ref{gongming}. Then \begin{align}
\label{1/2}
    m_0=\frac{n-1}{n}\frac{1}{\mu_2(\Omega)}\frac{P^2 (\Omega)}{|\Omega|}(\approx \frac{2\pi}{3.39}R^2\, \mbox{when $n=2$}),
\end{align}
where $\mu_2$ is the second eigenvalue of Neumann Laplacian on $\Omega$. 
\end{theorem} 
Hence from \eqref{1/2}, we know that $m_0$ is proportional to the volume of $\Omega$, and hence it cannot be small if the volume of $\Omega$ is not.

We remark that \eqref{1/2} has been proved in our recent paper \cite{HLL1}, but it was obtained as a byproduct from the domain variational point of view. Even though the proof there is interesting and implies other consequences such as stability of ball shape with prescribed volume, it is extremely complicated even in the case of $n=2$, since the argument involves second shape derivatives and requires choosing appropriate path of variation. Readers may be lost in the long computation and cannot see the intuition.  In this present paper, we will give a rather simple and direct proof of \eqref{1/2}, and especially in two dimensions, $2\pi$ can be easily visualized.

\medskip

On Question \ref{ques3}, the intuition is that ball should be the least possible shape on which the concentration breaking occurs, and hence the answer to Question \ref{ques3} should be positive. This is justified in the following theorem, and for simplicity we state for $C^1$ or convex domains.

\begin{theorem}
\label{q2main}
Let $\Omega \subset \mathbb{R}^n$ be a bounded $C^1$ or convex domain. Then there exists $m_0>0$ such that when $m<m_0$, $u_m$ must vanish on a subset of $\partial \Omega$ with positive $\mathscr{H}^{n-1}$ measure.
\end{theorem}

We also note that the proof of both Theorem \ref{concentration} and Theorem \ref{q2main} also requires us to prove global regularity of minimizers. Concerning this direction, so far the best regularity result we have is the following.
\begin{theorem}
\label{global}
Let $\Omega$ be a bounded Lipschitz domain in $\mathbb{R}^n$. If $u_m$ is a minimizer to \eqref{1problem} with $\int_{\Omega}u_m\, dx=1$, or a minimzier to \eqref{2problem} with $\int_{\Omega}u_m^2\, dx=1$, then there exists $\epsilon=\epsilon(\Omega)>0$ such that
\begin{align}
    \label{estigrauintro}
\Vert  u_m \Vert_{W^{1,p}(\Omega)} \le C(m,p,\Omega),
\end{align}for any $p \in (\frac{3}{2}-\epsilon,3+\epsilon)$. Moreover, minimizers to \eqref{1problem} are H\"older continuous up to the boundary for any dimension $n$. If $\Omega$ is $C^1$ or convex, then the range of $p$ in \eqref{estigrauintro} can be $(1,\infty)$.
\end{theorem}

As a consequence, Theorem \ref{concentration} and Theorem \ref{q2main} can be suitably extended to bounded Lipschitz domains, see more details in later sections.
\medskip

Last, we give some comments of our proof of the main results in this paper. The proof of Theorem \ref{concentration} is a consequence of uniqueness (up to constant factor) of solution to \eqref{1problem} and Serrin's overdetermined symmetry results \cite{Serrin}. The proof of Theorem \ref{exact} uses some elementary properties of Bessel functions and decomposition of Laplacian operator. The proof of Theorem \ref{q2main} is based on perturbation argument and comparison results of Laplacian eigenvalues with Dirichlet, Neumann and mean zero boundary conditions. Finally, the proof of \eqref{estigrauintro} in Theorem \ref{global} on global $W^{1,p}$ regularity of minimizers is essentially due to \cite{FMM98} and \cite{GS10}, and the global H\"older regularity for the first problem is due to \cite{KP93}, but the successful application of the references requires a regularization argument in the first problem and variational inequalities in the second problem, both of which are common techniques to deal with non-$C^1$ functionals in variational problems.

\medskip
\subsection*{\textbf{Further remarks on the breaking threshold in the second problem}}
\medskip

We first remark that the concentration breaking threshold given in \eqref{1/2} is also the stability breaking threshold for domains with prescribed volume in $\mathbb{R}^2$. More precisely, if $m_0$ is given by \eqref{1/2}, then it is shown in \cite{BBN} that when $m<m_0$, the disk cannot be a stationary shape to \eqref{2problem}, while when $m>m_0$, it is shown in \cite{HLL1} that the disk must be a local minimizer. 

Also, by Theorem \ref{q2main}, we can define the following positive number
\begin{align*}
    m_0(\Omega):=\sup\{a>0: \mbox{when $m<a$, solution to \eqref{2problem} must vanish on some portion of $\partial \Omega$} \}.
\end{align*}
Numerical results suggest for a wide class of symmetric domains, if we prescribe the volume of $\Omega$, then the lower bound for $m_0(\Omega)$ is exactly given by the right hand side of \eqref{1/2}, and in addition to balls in any dimension, the lower bound can also be attained at regular polygons in $\mathbb{R}^2$ with $k$ sides, $k \ge 4$. These results can motivate new isoperimetric type inequalities, and thus make the study of concentration breaking on \eqref{2problem} even more interesting beyond the physical application in thermal insulation. For more results on this aspect, we refer to  our forthcoming paper \cite{HLL2}.

We also note that Theorem \ref{q2main} only gives qualitative result, and it would be interesting to give a quantitative result similar to what we have done in Theorem \ref{concentration} and Theorem \ref{exact}. More precisely, let $\Omega^*$ be the ball with the same volume as that of $\Omega$. By Theorem \ref{gongming} and Theorem \ref{exact}, when $m<\frac{1}{2\mu_2(\Omega^*)}\frac{P^2 (\Omega^*)}{|\Omega^*|}$, the optimal insulation cannot concentrate on the whole boundary of the ball $\Omega^*$. Intuitively, there is no reason that when $m<\frac{1}{2\mu_2(\Omega^*)}\frac{P^2 (\Omega^*)}{|\Omega^*|}$, the optimal insulation would concentrate everywhere on $\partial \Omega$. Unfortunately, this still leaves open. 

\medskip

\textbf{Outline of the paper:}
In section 2, we study global regularity of minimizers to both optimization problems \eqref{1problem} and \eqref{2problem}. In section 3, we prove Theorem \ref{concentration} and its extension to Lipschitz domains stated in Theorem \ref{jiaqiang}. In section 4, we prove Theorem \ref{q2main} and discuss the case when $m$ is suffciently large stated in Proposition \ref{mlarge}. In section 5, we prove Theorem \ref{exact}. In section 6, we give some comments on the asymptotic behavior of solutions to \eqref{1problem} and \eqref{2problem} when $m$ is approaching to $0$. In section 7, we propose another open question motivated in the study of thermal insulation problems.

\section{Global Regularity of minimizers on Lipschitz domains}
In both minimization problems \eqref{1problem} and \eqref{2problem}, minimizers could vanish on a subset of $\partial \Omega$, and thus it is not convenient to do the first variation due to that the boundary term $\int_{\partial \Omega}|\cdot |\, d\sigma$ involved in the functionals is not $C^1$-regular. In order not only to derive the appropriate form of Euler-Lagrange equation but also to study the boundary concentration of minimizers via perturbation argument, technically we should at least require that any minimizer $u_m$ is continuous up to the boundary, so that we can decompose $\partial \Omega$ into two parts: $\{|u_m|\ne 0\}\cap \partial \Omega$ and $\{|u_m|=0\} \cap \partial \Omega$, and the former one is open.

Global regularity of minimizers certainly depends on the regularity of $\partial \Omega$. In \cite{BBN1}-\cite{Buttazzo} and \cite{DLW} where the first problem \eqref{1problem} is studied, $\Omega$ is assumed to be sufficiently regular, but that how much regularity can ensure desired regularity of minimizers is not discussed. In \cite{BBN} and \cite{HLL1} where the second problem \eqref{2problem} is studied, the regularity of $\Omega$ is also not addressed since there the main focus is on boundary distribution of minimizers on radial domains and smooth variations at ball shape. 

In this section, we will prove certain regularity results for minimizers in bounded Lipschitz domains, since the Lipschitz condition on the domain is a natural assumption for the existence of minimizers, due to Poincar\'e inequalities. 

Let us first recall a well known solvability result for the following Possion equation with Neumann boundary condition, \begin{align}
    \label{Neumannproblem}
\begin{cases}
-\Delta u=f \quad &\mbox{in $\Omega$}\\
\frac{\partial u}{\partial \nu}=g \quad &\mbox{on $\partial \Omega$}.
\end{cases}
\end{align}
Here $\Omega$ is a bounded Lipschitz domain, $f \in  L^p_{-1,0}(\Omega)$, the dual of $W^{1,q}(\Omega)$, and $g \in B^p_{-\frac{1}{p}}(\partial \Omega)$, the dual of the Besov space $B^q_{\frac{1}{p}}(\partial \Omega)$, where $q=\frac{p}{p-1}$. By $u$ being a solution to \eqref{Neumannproblem}, we mean that for any $\phi\in W^{1,q}(\Omega)$,
\begin{align}
    \label{solvesense}
\int_{\Omega}\nabla u \nabla \phi \, dx-<g, tr(\phi)>=<f,\phi>,
\end{align}where $tr(\phi)$ denotes the trace of $\phi$ on $\partial \Omega$, belonging to $W^{1-\frac{1}{q},q}(\partial \Omega)=W^{\frac{1}{p},q}(\partial \Omega)$, which is also denoted in some literature as $B^q_{\frac{1}{p}}(\partial \Omega)$.

Then the following theorem is due to \cite{FMM98} for Lipschitz and $C^1$ domains and \cite{GS10} for convex domains.
\begingroup
\renewcommand{\thetheorem}{\Alph{theorem}}
\begin{theorem}
\label{cite1}
Let $\Omega$ be a bounded Lipschitz domain in $\mathbb{R}^n$. There exists a
positive number $\epsilon=\epsilon(\Omega)>0$ with the following significance. If $p \in (\frac{3}{2}-\epsilon, 3+\epsilon)$, then for any $f \in L^p_{-1,0}(\Omega)$ and $g \in B^p_{-\frac{1}{p}}(\partial \Omega)$ with $\int_{\Omega}f \, dx=-\int_{\Omega}g\, dx$, the Neumann problem \eqref{Neumannproblem} admits a unique (up to constants) solution $u$. Moreover, $\nabla u$ satisfies the following estimate
\begin{align}
    \label{estigrau}
\Vert \nabla u \Vert_{L^p(\Omega)} \le C(p,\Omega)\left(\Vert f\Vert_{ L^p_{-1,0}(\Omega)}+\Vert g \Vert_{B^p_{-\frac{1}{p}}(\partial \Omega})\right).
\end{align}
If $\Omega$ is $C^1$ or convex, then the above conclusion is also true for any $p \in (1,\infty)$.
\end{theorem}
\endgroup

Now we state the existence, uniqueness and regularity result for minimizers of \eqref{1problem}.

\begin{proposition}
\label{1es}
Let $m>0$ and $\Omega$ be a connected Lipschitz domain. Then up to constant factor, minimizer to \eqref{1problem} is unique. Moreover, if $u_m$ is a minimizer to \eqref{1problem} with the constraint
\begin{align}
    \label{prescribtion1}
\int_{\Omega}u_m\, dx=1,    
\end{align}
then there exists $\epsilon=\epsilon(\Omega)>0$ such that
\begin{align}
    \label{estigrau}
\Vert u_m \Vert_{W^{1,p}(\Omega)} \le C(p,m,\Omega),
\end{align}for any $p \in (\frac{3}{2}-\epsilon,3+\epsilon)$. If $\Omega$ is $C^1$ or convex, then the range of $p$ can be $(1,\infty)$.
\end{proposition}

\begin{proof}
Let 
\begin{align}
\label{Tm}
    T_m(u)=\int_{\Omega}|\nabla u|^2\, dx+\frac{1}{m}\left(\int_{\partial \Omega}|u|\, d\mathscr{H}^{n-1}\right)^2, 
\end{align}
and for a small positive number $0<\delta<1$ we let
\begin{align}
\label{tmd}
    T_m^{\delta}(u)=\int_{\Omega}|\nabla u|^2\, dx+\frac{1}{m}\left(\int_{\partial \Omega}\sqrt{u^2+\delta^2}\, d\mathscr{H}^{n-1}\right)^2.
\end{align}Also, we let
\begin{align}
\label{dd}
   t_m^{\delta}=\inf\{T_m^{\delta}(u):u \in H^1(\Omega), \int_{\Omega}u\, dx=1\}.
\end{align}
By direct method and the Poincar\'e type inequality
\begin{align}
\label{poin0}
    \int_{\Omega} u^2 \le C(\Omega)\left(\int_{\Omega}|\nabla u|^2\, dx+\left(\int_{\Omega}u\, dx\right)^2\right),
\end{align}the infimum in \eqref{dd} can be attained at a function $u_m^{\delta}$  with $\int_{\Omega}u^{\delta}_m\, dx=1$. Hence $u_m^{\delta}$ satisfies the Euler-Lagrange equation
$$\begin{cases}-\Delta u_m^{\delta}=-\frac{\int_{\partial \Omega}g_m^\delta \, d\sigma}{|\Omega|} & \ {\rm{in}}\ \Omega,\\
\displaystyle\frac{\partial u_m^{\delta}}{\partial\nu}=g_m^\delta:=-\big(\frac{1}{m}\int_{\partial\Omega}\sqrt{(u_m^\delta)^2+\delta^2}\,d\sigma\big)\frac{u_m^{\delta}}{\sqrt{(u_m^\delta)^2+\delta^2}} & 
\ {\rm{on}}\ \partial\Omega.
\end{cases}
$$
Since for any $\delta \in (0,1)$, 
\begin{align*}
   \int_{\Omega}|\nabla u_m^\delta|^2\, dx+\frac{1}{m}\left(\int_{\partial \Omega}\sqrt{(u_m^\delta)^2+\delta^2}\, d\mathscr{H}^{n-1}\right)^2= t_m^\delta \le T_m^\delta(\frac{1}{|\Omega|})\le C(m,\Omega),
\end{align*} and by \eqref{poin0} we have
$$
\|u_m^\delta\|_{H^1(\Omega)}\le C(m,\Omega), \, \forall \, 0<\delta\le 1.$$
Since 
\begin{align*}
    |g_m^{\delta}|\le \frac{1}{m}\int_{\partial \Omega}\sqrt{(u_m^\delta)^2+\delta^2}\, d\mathscr{H}^{n-1}\le C(\Omega)\|u_m^\delta\|_{L^2(\partial \Omega)},
\end{align*}
by standard Sobolev trace embedding theorem on Lipschitz domains, we have
$$\|g_m^\delta\|_{L^\infty(\partial\Omega)}\le C(m,\Omega), \ \forall 0<\delta< 1.$$
Hence by Theorem \ref{cite1}, $u_m^\delta \in W^{1,p}(\Omega)$, with
\begin{align}
\label{deltaestimate}
    \|u_m^\delta\|_{W^{1,p}(\Omega)}\le C(m, \Omega), \ \forall 0<\delta< 1,
\end{align}where $p \in (\frac{3}{2}-\epsilon, 3+\epsilon)$ for some $\epsilon=\epsilon(\Omega)>0$ if $\Omega$ is Lipschitz, and $p \in (1,\infty)$ if $\Omega$ is $C^1$ or convex.

Hence we may assume, after taking a possible subsequence of $\delta \rightarrow 0$, that there exists $v_m\in W^{1,p}(\Omega)$ such that
$$u_m^\delta\rightharpoonup v_m \ {\rm{in}}\ W^{1,p}(\Omega), \ \forall\, \mbox{$p$ belongs to the range in Theorem \ref{cite1}}.$$
Hence $\int_{\Omega}v_m\, dx=1$. Now we want to show that $v_m$ is also a minimizer to \eqref{1problem}, and clearly it suffices to show that  $T_m(v)\le T_m(w)$ for any function $w\in H^1(\Omega)$ with $\int_{\Omega}w\, dx=1$. This actually follows from the lower semicontinuity that
$$T_m(v_m)\le\liminf_{\delta\rightarrow 0}T_m^\delta(u_m^\delta)\le \liminf_{\delta\rightarrow 0} T_m^\delta(w)
=T_m(w).$$
Now let us prove $v_m=u_m$. Indeed, if $v_m \ne u_m$, then $\tilde{u}_m:=\frac{v_m+u_m}{2}$ also satisfies $\int_{\Omega}u_m\, d\sigma=1$. Then $v_m=u_m$ directly follows from  $$T_m(\tilde{u}_m)\ge \frac{T_m(u_m)+T_m(v_m)}{2},$$and the connectivity of $\Omega$. For the details of this step we refer to \cite[Proposition 2.1]{BBN}. Hence we have shown existence and uniqueness of minimizers to \eqref{1problem} with constraint \eqref{prescribtion1}, and the $W^{1,p}$ regularity \eqref{estigrau} follows from \eqref{deltaestimate} and uniqueness.
\end{proof}

\begin{remark}
\label{Kenigpoint}
In fact, for any bounded Lipschitz domain in $\mathbb{R}^n$, the minimizer $u_m$ in Proposition \ref{1es} actually belongs to the H\"older space $C^{\alpha}(\bar{\Omega})$ for some $\alpha \in (0,1)$ depending on $\Omega$. 
\end{remark}

\begin{proof}
This is essentially a consequence of Kenig-Pipher\cite[Corollary 2.14 and Theorem 2.23]{KP93}, where the authors introduce the Neumann function for elliptic equation of divergence form on balls with bounded measurable coefficients, and proved H\"older continuity of Neumann function and a representation formula for generalized solutions. As a consequence, global H\"older estimate is valid for harmonic function with bounded Neumann data on the unit ball. Then by a change of variables which is explained for example in \cite{Kenig}, this is also valid on starlike Lipschitz domains. For the case of general bounded Lipschitz domains similar arguments applies too, or a localization argument would also work to reduce to star-like Lipschitz domanis. Hence we have H\"older estimate for harmonic function with bounded Neumann data on Lipschitz domains. For the inhomogeneous Possion equation \eqref{Neumannproblem} on a bounded Lipschitz domain, if $f \in L^p(\Omega)$ for some $p>n$ and $g \in L^{\infty}(\partial \Omega)$, then by extending $f$ to be zero outside, and according to $W^{2,p}$ estimate and Sobolev embedding, by adding a $C^{1,\alpha}$ function, we reduce to the harmonic function case with bounded Neumann data. Hence we still have H\"older estimate for solution to \eqref{Neumannproblem} with $f \in L^p(\Omega)$ ($p>n$) and $g \in L^{\infty}(\partial \Omega)$. Hence $u_m$ has H\"older regularity on $\bar{\Omega}$, by exact same argument as in the proof of Proposition \ref{1es}.
\end{proof}

\medskip

Different from the first one, minimizers to the second problem are usually not unique even up to a constant factor. Nevertheless, we still have:
\begin{proposition}
\label{2es}
Let $m>0$ and $\Omega$ be a bounded Lipschitz domain. Then there exists a minimizer $u_m$ to \eqref{2problem}. Moreover, if $u_m$ satisfies
\begin{align}
    \label{prescribtion2}
\int_{\Omega}u_m^2\, dx=1,    
\end{align}
then there exists $\epsilon=\epsilon(\Omega)>0$ such that
\begin{align}
    \label{estigrau2}
\Vert u_m \Vert_{W^{1,p}(\Omega)} \le C(m,p,\Omega),
\end{align}for any $p \in (\frac{3}{2}-\epsilon,3+\epsilon)$. If $\Omega$ is $C^1$ or convex, then the range of $p$ can be $(1,\infty)$.
\end{proposition}

\begin{proof}
The existence of minimizers follows from direct method and the Poincar\'e inequality
\begin{align}
\label{poin}
    \int_{\Omega} u^2 \le C(\Omega)\left(\int_{\Omega}|\nabla u|^2\, dx+\left(\int_{\partial \Omega}|u|\, dx\right)^2\right).
\end{align}
Let $\lambda_m$ be the functional in \eqref{2problem}, that is,
\begin{align*}
    \lambda_m(u)=\frac{\int_{\Omega}|\nabla u|^2\, dx+\left(\frac{1}{m}\int_{\partial \Omega}|u|\, d\mathscr{H}^{n-1}\right)^2}{\int_{\Omega}u^2\, dx}.
\end{align*}We also let $u_m$ be a minimizer satisfying \eqref{prescribtion2} and let $\lambda_m=\lambda_m(u_m)$. For any $\phi \in H^1(\Omega)$, using \eqref{prescribtion2}, we have
\begin{align}
\label{chafen}
    \lambda_m(u_m+t\phi)-\lambda_m
    =&\frac{2t\int_{\Omega}\nabla u_m \nabla \phi \, dx+t^2\int_{\Omega}|\nabla \phi|^2\, dx}{\int_{\Omega}(u_m+t\phi)^2\, dx}-\frac{\left(2t\int_{\Omega}u_m\phi\, dx+t^2\int_{\Omega}\phi^2\, dx\right)\lambda_m}{\int_{\Omega}(u_m+t\phi)^2\, dx}\nonumber\\
    &+\frac{\frac{1}{m}\left(\int_{\partial \Omega}|u_m+t\phi|\, d\sigma+\int_{\partial \Omega}|u_m|\, d\sigma\right)\left(\int_{\partial \Omega}|u_m+t\phi|\, d\sigma-\int_{\partial \Omega}|u_m|\, d\sigma\right)}{\int_{\Omega}(u_m+t\phi)^2\, dx}.
\end{align}
Since \begin{align*}
    \int_{\partial \Omega}|u_m+t\phi|\, d\sigma-\int_{\partial \Omega}|u_m|\, d\sigma\le |t|\int_{\partial \Omega}|\phi|\, d\sigma,
\end{align*}from \eqref{prescribtion2} and \eqref{chafen} we have that
\begin{align}
\label{shou1}
    \liminf_{t\rightarrow 0^+}\frac{\lambda_m(u_m+t\phi)-\lambda_m}{t}\le 2\int_{\Omega}\nabla u_m \nabla \phi \, dx-2\lambda_m\int_{\Omega}u_m\phi\, dx+\frac{2}{m}\left(\int_{\partial \Omega}|u_m|\, d\sigma \right)\left(\int_{\partial \Omega}|\phi|\, d\sigma\right),
\end{align}and that
\begin{align}
\label{shou2}
    \liminf_{t\rightarrow 0^-}\frac{\lambda_m(u_m+t\phi)-\lambda_m}{t}\ge 2\int_{\Omega}\nabla u_m \nabla \phi \, dx-2\lambda_m\int_{\Omega}u_m\phi\, dx-\frac{2}{m}\left(\int_{\partial \Omega}|u_m|\, d\sigma \right)\left(\int_{\partial \Omega}|\phi|\, d\sigma\right).
\end{align}Now by \eqref{shou1}-\eqref{shou2} and $\lambda_m(u_m+t\phi)\ge \lambda_m$, we have
\begin{align*}
\Big|\int_{\Omega}\nabla u_m\nabla \phi \, dx-\lambda_m\int_{\Omega}u_m\phi \, dx\Big| \le \frac{1}{m}\left(\int_{\partial \Omega}|u_m|\, d\sigma\right)\left( \int_{\partial \Omega}|\phi|\, d\sigma\right).
\end{align*}
Hence there exists $g \in L^{\infty}(\partial \Omega)$ with
\begin{align}
\label{g1}
    \|g\|_{L^{\infty}(\partial \Omega)} \le \frac{1}{m}\int_{\partial \Omega}|u_m|\, d\sigma, 
\end{align}
such that
\begin{align*}
    \int_{\Omega}\nabla u_m \nabla \phi \, dx-\lambda_m\int_{\Omega}u_m\phi\,dx-\int_{\partial \Omega}g\phi\, d\sigma=0.
\end{align*}
Hence $u_m$ is a solution to the following equation
\begin{align*}
    \begin{cases}
-\Delta u=\lambda_m u \quad &\mbox{in $\Omega$}\\
\frac{\partial u}{\partial \nu}=g \quad &\mbox{on $\partial \Omega$}.
\end{cases}
\end{align*}
Since \begin{align}
    \label{lizhi}
\int_{\Omega}|\nabla u_m|^2\, dx+\frac{1}{m}\left(\int_{\partial \Omega}|u_m|\, d\sigma \right)^2 =\lambda_m(u_m)\le \lambda_D(\Omega),
\end{align} where $\lambda_D(\Omega)$ is the first eigenvalue of Dirichlet Laplacian on $\Omega$, we have
\begin{align}
    \label{ggg}
\|u_m\|_{H^1(\Omega)}\le C(\Omega).  
\end{align}
Also, by \eqref{g1} and \eqref{lizhi}, we have
\begin{align*}
     \|g\|_{L^{\infty}(\partial \Omega)}\le\frac{1}{\sqrt{m}}\sqrt{\lambda_D(\Omega)}.
\end{align*}
Therefore, by Theorem \ref{cite1}, \eqref{ggg} and Sobolev embedding, we have 
\begin{align}
    \label{near}
\|u_m\|_{W^{1,p}(\Omega)} \le C(m,p,\Omega),
\end{align}
for $p=p_1=\min\{\bar{p}, 2^*\}$, where $\bar{p}$ is as in Theorem \ref{cite1} and $2^*=\frac{2n}{n-2}$. Actually \eqref{near} also holds for $p_{k+1}=\min\{\bar{p}, p_k^*\}$, where $p_k^*=\frac{np_k}{n-p_k}$ and $p_0=2$. Sending $k \rightarrow \infty$, we conclude our proposition.
\end{proof}

Theorem \ref{global} is then a consequence of Proposition \ref{1es}, Remark \ref{Kenigpoint} and Proposition \ref{2es}.

\section{Concentration breaking for maximizing heat content}
In this section, we prove concentration breaking result for the first problem \eqref{1problem}. Recall that this corresponds to the thermal insulation problem of maximizing heat content.

We first need the following lemma.
\begin{lemma}
\label{jianhua}
Let $u$ be a weak solution to the following equation.
\begin{align}
    \label{dengjiau0}
\begin{cases}
-\Delta u=1 \quad &\mbox{in $\Omega$}\\
\frac{\partial u}{\partial \nu}=-\frac{1}{m}\int_{\partial \Omega}u d\sigma &\mbox{on $\partial \Omega$}.
\end{cases} 
\end{align}
If $u\ge 0$ on $\partial \Omega$, then $u$ must be a minimizer to \eqref{1problem}.
\end{lemma}
\begin{proof}
Let $T_m$ be the functional as in the proof of Proposition \ref{1es}. To show $u$ is a minimizer, if suffices to show that \begin{align}
    \label{t}
T_m(u)\le T_m(v),    
\end{align}
for  any $v \in H^1(\Omega)$ with 
\begin{align*}
    \int_{\Omega}u\, dx=\int_{\Omega}v\, dx.
\end{align*}
In fact, multiplying $\eqref{dengjiau0}_1$ by $(u-v)$, taking integral over $\Omega$, and integration by parts lead to
\begin{align}
\label{t1}
    \int_{\Omega}|\nabla u|^2\, dx+\frac{1}{m}\left(\int_{\partial \Omega}u\, d\sigma\right)^2=\int_{\Omega}\nabla u\nabla v\, dx+\frac{1}{m}\left(\int_{\partial \Omega}u\,d\sigma\right)\left( \int_{\partial \Omega}v\, d\sigma\right).
\end{align}
The RHS above is bounded above by
\begin{align}
    \label{t2}
\frac{1}{2}\int_{\Omega}|\nabla u|^2\, dx+\frac{1}{2}\int_{\Omega}|\nabla v|^2\, dx+\frac{1}{2m}\left(\int_{\partial \Omega}u\,d\sigma\right)^2+\frac{1}{2m}\left(\int_{\partial \Omega}v\,d\sigma\right)^2.    
\end{align}
Hence \eqref{t1}-\eqref{t2} imply
\begin{align*}
    \int_{\Omega}|\nabla u|^2\, dx+\frac{1}{m}\left(\int_{\partial \Omega}u\,d\sigma\right)^2 \le  \int_{\Omega}|\nabla v|^2\, dx+\frac{1}{m}\left(\int_{\partial \Omega}v\,d\sigma\right)^2.
\end{align*}Since $u\ge 0$ on $\partial \Omega$, we immediately have \eqref{t}. This finishes the proof.
\end{proof}

Now we are ready to prove Theorem \ref{concentration}. In fact, we will prove a stronger version on domains with weaker regularity.

\begin{theorem}
\label{jiaqiang}
For any bounded Lipschitz domain $\Omega \subset \mathbb{R}^n$, there exists $m_1=m_1(\Omega)\ge 0$ such that when $m>m_1$, any minimizer to \eqref{1problem} must be nowhere vanishing on $\partial \Omega$, while if $m_1>0$ and $m \in (0,m_1)$, then any minimizer to \eqref{1problem} must be vanishing on a subset of $\partial \Omega$ with positive $\mathscr{H}^{n-1}$ measure. If $\Omega$ is further assumed to be $C^2$, then $m_1>0$ as long as $\Omega$ is not a ball.
\end{theorem}
\begin{proof}
Let $u_m$ be a solution to \eqref{1problem}, and we may assume that $u_m \ge 0$ in $\Omega$, otherwise we can take the absolute value and does not increase the minimizing functional. 

By Proposition \ref{1es} and Remark \ref{Kenigpoint}, $u_m$ belong to $C^{\alpha}(\bar{\Omega})$ for some $\alpha \in (0,1)$ when $\Omega$ is Lipschitz, and for any $\alpha \in (0,1)$ when $\Omega$ is $C^1$ or convex. Hence, by direct calculations, we obtain that up to a constant factor, $u_m$ satisfies the following boundary value problem
\begin{equation}\label{EulerLagrangeforu'}
\begin{cases}
-\Delta u=1 & \ {\rm{in}}\ \Omega,\\
\displaystyle\frac{\partial u}{\partial\nu}=-\frac{1}{m}\int_{\partial\Omega} u\,d\sigma & \ {\rm{on}}\ \partial\Omega\cap\{x: u(x)>0\},\\
\displaystyle\frac{\partial u}{\partial\nu}\ge-\frac{1}{m}\int_{\partial\Omega} u\,d\sigma & \ {\rm{on}}\ \partial\Omega\cap\{x: u(x)=0\}.
\end{cases}
\end{equation}

Let $u_m'$ be the solution to \eqref{dengjiau0},and $u_0$ be the solution to \begin{align}
    \label{dengjiau1}
\begin{cases}
-\Delta u=1 \quad &\mbox{in $\Omega$}\\
\frac{\partial u}{\partial \nu}=-\frac{|\Omega|}{P(\Omega)} &\mbox{on $\partial \Omega$.}
\end{cases} 
\end{align}
Note that $u_m'$ can be uniquely determined in terms of $u_0$ by the following relation:
\begin{align}
\label{um}
    u_m'=u_0+\frac{m|\Omega|}{P^2(\Omega)}-\frac{1}{P(\Omega)}\int_{\partial \Omega}u_0d\sigma.
\end{align}
Let
\begin{align}
\label{deltaomega}
\delta_{\Omega}=\frac{1}{P(\Omega)}\int_{\partial \Omega}u_0d\sigma-\min_{\partial \Omega}u_0, 
\end{align}which is well defined since $u_0 \in C^{\alpha}(\bar{\Omega})$ for some $\alpha \in (0,1)$. Note that the constant $\delta_{\Omega}$ only depends on the shape of $\Omega$. 

By \eqref{um}, we know that
\begin{align}
    \label{deltaomega'}
\min_{\partial \Omega}u_m'=-\delta_{\Omega}+\frac{m|\Omega|}{P^2(\Omega)}.
\end{align}

Hence from \eqref{deltaomega'} we know that as $m>m_1:= \delta_{\Omega}\frac{P^2(\Omega)}{|\Omega|}$, $u_m' > 0$ on $\partial \Omega$. Hence by Lemma \ref{jianhua}, $u_m'$ is a minimizer to \eqref{1problem}. By uniqueness property shown in Proposition \ref{1es}, it entails that $u_m'=u_m$ up to a constant factor, and hence $u_m>0$ everywhere on $\partial \Omega$. If $\delta_{\Omega}>0$ and $0<m<\delta_{\Omega}\frac{P^2(\Omega)}{|\Omega|}$, then according to \eqref{deltaomega'}, $\min_{\partial \Omega}u_m'<0$ and hence \eqref{EulerLagrangeforu'} does not admit a solution which is everywhere positive on $\partial \Omega$. This means that $u_m$ must be vanishing on a subset of $\partial \Omega$ with positive $\mathscr{H}^{n-1}$ measure. 

By the famous result of Serrin\cite{Serrin}, we know that when $\Omega$ is $C^2$, $u_0$ cannot be a constant on $\partial \Omega$, unless $\Omega$ is a ball, and hence $\delta_{\Omega}>0$ as long as $\Omega$ is not a ball. 
\end{proof}

\section{Concentration breaking for minimizing temperature decay}
In this section, we will prove concentration breaking result for the second problem \eqref{2problem}. Recall that this corresponds to the thermal insulation problem of minimizing temperature decay.

Before proving Theorem \ref{q2main}, we first prove the following simple lemma on comparison of eigenvalues with different boundary condition.
\begin{lemma}
\label{com}
Let $\Omega$ be a bounded Lipschitz domain in $\mathbb{R}^n$ and define
\begin{align}
\label{kappa1}
    \kappa_1(\Omega):=\inf\Big\{\frac{\int_{\Omega}|\nabla u|^2 \, dx}{\int_{\Omega}u^2\, dx}: \int_{\partial \Omega}u\, d\sigma=0 \Big\}.
\end{align}
Then $\kappa_1(\Omega) \le \mu_2(\Omega)$.
\end{lemma}
\begin{proof}
Let $v$ be an eigenfunction of the second Neumann Laplacian eigenvalue, and let $$\bar{v}=\frac{1}{P(\Omega)}\int_{\partial \Omega}v\, d\sigma\quad \mbox{and}\quad  w=v-\bar{v}.$$ Hence $\int_{\partial \Omega} w\,d\sigma=0$, and 
\begin{align}
\label{guiju}
    \frac{\int_{\Omega}|\nabla w|^2\, dx}{\int_{\Omega}w^2\, dx}=\frac{\int_{\Omega}|\nabla v|^2\, dx}{\int_{\Omega}(v^2+\bar{v}^2)\, dx}\le \mu_2(\Omega),
\end{align}where we have used that $\int_{\Omega}v\,dx=0$. Since $w$ is a valid trial function for \eqref{kappa1}, the lemma is proved.
\end{proof}

\begin{remark}
\label{zhu}
It can be seen from the proof that $\kappa_1=\mu_2$ only if the trace of second fundamental modes of Neumann Laplacian has mean zero. So far we do not know whether this is also a suffcient condition, but a different necessary and sufficient condition is given in our forthcoming paper \cite{HLL2} from the study of another more theoretical eigenvalue problem. In particular, $\kappa_1=\mu_2$ when the domain is a ball, rectangle or equilateral triangle.
\end{remark}

Now we are ready to prove Theorem \ref{q2main}.

\begin{proof}[Proof of Theorem \ref{q2main}]
Let $u_m$ be as in the hypothesis of Theorem \ref{q2main}, that is, $u_m$ is a function where the following infimum is attained:
\begin{align}
    \label{lambdam'}
\lambda_m(\Omega)=\inf\Big\{\frac{\int_{\Omega}|\nabla u|^2dx+\frac{1}{m}\left(\int_{\partial \Omega}|u|d\sigma\right)^2}{\int_{\Omega}u^2dx}:u\in H^1(\Omega)\Big\}.
\end{align}
We may assume that $u_m \ge 0$.

By Lemma \ref{com}, $\kappa_1(\Omega)\le \mu_2(\Omega)$, and then by Friedlander's Theorem, $\kappa_1(\Omega)$ is strictly less than $\lambda_D(\Omega)$, the first eigenvalue of Dirichlet Laplacian. This is also a consequence of Faber-Krahn inequality and Szeg\"o-Weinberger inequality. Since as $m\rightarrow 0$, $\lambda_m(\Omega)$ tends to $\lambda_D(\Omega)$, and as $m\rightarrow \infty$, $\lambda_m(\Omega)$ tends to $0$, we can find $m_0>0$ such that $\lambda_{m_0}(\Omega)=\kappa_1(\Omega)$.

Let $w$ be the function such that the infimum in \eqref{kappa1} is achieved. WLOG we assume that $\int_{\Omega}u_m^2\, dx=\int_{\Omega}w^2\, dx=1$. We argue by contradiction. Suppose that $u_m>0$ everywhere on $\partial \Omega$ for some $m<m_0$, then consider $u_m+\epsilon w$ as a trial function for \eqref{lambdam'}. Since such $u_m$ satisfies
\begin{align}
    \label{EL2}
\begin{cases}
-\Delta u=\lambda_m(\Omega) u \quad &\mbox{in $\Omega$}\\
\frac{\partial u}{\partial \nu}=-\frac{1}{m}\int_{\partial \Omega}u\, d\sigma &\mbox{on $\partial \Omega$}
\end{cases}    
\end{align}
Since $\Omega$ is $C^1$ or convex, by Proposition \ref{2es}, $u_m \in C^{\alpha}(\bar{\Omega})$ for any $\alpha \in (0,1)$. Since $w$ satisfies
\begin{align}
\label{wequation}
    \begin{cases}
    -\Delta w=\kappa_1(\Omega)w\quad &\mbox{in $\Omega$}\\
    \frac{\partial w}{\partial \nu}=-\frac{\kappa_1(\Omega)}{P(\Omega)}\int_{\Omega}w\, dx\quad &\mbox{on $\partial \Omega$}\\
    \int_{\partial \Omega}w\, d\sigma=0.
    \end{cases}
\end{align}
By Theorem \ref{cite1} and bootstrap argument, $w$ also belongs to $C^{\alpha}(\bar{\Omega})$ for any $\alpha \in (0,1)$. Hence $\epsilon$ can be chosen to be sufficiently small such that $u_m+\epsilon w>0$ on $\partial \Omega$. Let 
\begin{align}
\label{jihao}
    \lambda_m(f,\Omega):=\frac{\int_{\Omega}|\nabla f|^2\,dx+\frac{1}{m}\left(\int_{\partial \Omega}|f|\,d\sigma\right)^2}{\int_{\Omega}f^2\,dx}.
\end{align}
On the one hand, since $u_m$ is a minimizer, $\lambda_m(u_m+\epsilon w, \Omega)\ge \lambda_m(\Omega)$.
On the other hand, by \eqref{EL2} and \eqref{wequation}, we have that $\int_{\Omega}\nabla u_m \nabla w \,dx=\lambda_m(\Omega)\int_{\Omega}u_mw\, dx$. Also, since $\lambda_m(\Omega)>\kappa_1(\Omega)$ as $m<m_0$, we have
\begin{align*}
    \lambda_m(u_m+\epsilon w,\Omega)=&\frac{\int_{\Omega}|\nabla u_m|^2dx+2\epsilon \lambda_m(\Omega) \int_{\Omega}u_mw\,dx+\epsilon^2\int_{\Omega}|\nabla w|^2 \, dx +\frac{1}{m}\left(\int_{\partial \Omega}u_m\,d\sigma\right)^2}{\int_{\Omega}u_m^2dx+2\epsilon \int_{\Omega}u_mw\, dx+\epsilon^2\int_{\Omega}w^2\, dx}\\
    =&\frac{\lambda_m(\Omega)+2\epsilon \lambda_m(\Omega) \int_{\Omega}u_mw\,dx+\epsilon^2\kappa_1(\Omega)}{1+2\epsilon \int_{\Omega}u_mw\, dx+\epsilon^2}\\
    <& \frac{\lambda_m(\Omega)+2\epsilon \lambda_m(\Omega) \int_{\Omega}u_mw\,dx+\epsilon^2\lambda_m(\Omega)}{1+2\epsilon \int_{\Omega}u_mw\, dx+\epsilon^2}=\lambda_m(\Omega).
\end{align*}This leads to a contradiction. Therefore, we conclude that when $m<m_0$, $u_m$ must vanish on a subset of $\partial \Omega$ with positive $\mathscr{H}^{n-1}$ measure.
\end{proof}

The next proposition shows that if the total amount of material is large enough, then the best insulation should cover the whole boundary.
\begin{proposition}
\label{mlarge}
Let $u_m$ be a minimizer to \eqref{2problem} and $\Omega$ is a $C^1$ or convex domain. If $m$ is sufficiently large, then $u_m$ must be vanishing nowhere on $\partial \Omega$.
\end{proposition}
\begin{proof}
WLOG, we assume that $\int_{\Omega}u_m^2 \, dx=1$. Hence by Proposition \ref{2es}, for any $p\in (1,\infty)$, the $W^{1,p}$ norm of $u_m$ on $\Omega$ is bounded above by a constant depending on $p$, $m$ and $\Omega$. Actually the proof of Proposition \ref{2es} says that 
$$\|u_m\|_{W^{1,p}(\Omega)} \le \frac{1}{\sqrt{m}}C(p,\Omega).$$
Hence as $m \rightarrow \infty$, $u_m$ has uniformly bounded $W^{1,p}(\Omega)$ norm for any fixed $p\in (1,\infty)$. By H\"older embedding, $u_m$ converges uniformly to a limit function. Note that the limit function has to be a nonzero constant, and hence $u_m$ must be nowhere vanishing on $\partial \Omega$ when $m$ is large.
\end{proof}

\begin{remark}
One can see that the proofs of Theorem \ref{q2main} and Proposition \ref{mlarge} go through as long as minimizers are continuous up to the boundary. Hence both Theorem \ref{q2main} and Proposition \ref{mlarge} are also valid for bounded Lipschitz domains in $\mathbb{R}^2$ and $\mathbb{R}^3$. In higher dimensions we believe this is also true, but so far we have not found a suitable argument to valid it.
\end{remark}







\section{On exact value of breaking thresholds for ball domains.}
Let $J_s(z)$ be the Bessel function of first kind with order $s$. Recall that (see \cite{OLBC})
\begin{align}
\label{1}
    J_s'(z)=-J_{s+1}(z)+\frac{sJ_s(z)}{z}
\end{align} and
\begin{align}
\label{2}
  J_s'(z)=J_{s-1}(z)-\frac{sJ_s(z)}{z}.
\end{align}

Before proving Theorem \ref{exact}, we first state a simple lemma in two dimensions, which is the most interesting case.
\begin{lemma}
\label{2bel}
Let $\Omega$ be a bounded smooth simply connected domain in $\mathbb{R}^2$, $m>0$, $\lambda_m=\lambda_m(\Omega)$ and $u_m$ be a nonnegative solution to \eqref{2problem}. Then
\begin{align}
\label{2beljuti}
   ( m\lambda_m-2\pi)\int_{\partial \Omega}u_m\, d\sigma=-m\int_{\partial \Omega}\frac{\partial^2 u_m}{\partial \nu^2}\, d\sigma.
\end{align}
\end{lemma}

\begin{proof}
First, by standard regularity theory, $u_m$ is a classical solution to \eqref{EL2}. Since on $\partial \Omega$ we have
\begin{align*}
    \Delta =\frac{\partial^2}{\partial \nu^2}+ H\frac{\partial}{\partial \nu}+\Delta_{\partial \Omega},
\end{align*}where $H$ is the mean curvature on $\partial \Omega$ and $\Delta_{\partial \Omega}$ is the Laplacian Beltrami operator on $\partial \Omega$, we have that
\begin{align}
\label{invH}
   \left( m\lambda_m-\int_{\partial \Omega}H\, d\sigma\right)\int_{\partial \Omega}u_m\, d\sigma=-m\int_{\partial \Omega}\frac{\partial^2 u_m}{\partial \nu^2} \, d\sigma.
\end{align}
Since $\Omega$ is simply connected, $\int_{\partial \Omega}H\, d\sigma=2\pi$. This toegether with \eqref{invH} implies \eqref{2beljuti}.
\end{proof}
We then have the following corollary, which validates Theorem \ref{exact} in two dimensions.
\begin{corollary}
Let $\Omega$ be a disk of radius $R$ in $\mathbb{R}^2$ and $m_0$ be as in Theorem \ref{gongming}. Then 
\begin{align*}
    m_0 \mu_2(\Omega)=2\pi.
\end{align*}
\end{corollary}
\begin{proof}
Since as $m>m_0$, $u_m$ is a radial solution, and the solution is unique up to a constant. Thus we may write $u_m=J_0(\sqrt{\lambda_m}r)$. Then by \eqref{1},
\begin{align*}
    \frac{d^2}{dr^2}u_m=-\lambda_m J_1'(\sqrt{\lambda_m}r).
\end{align*}
Since $\lambda_{m_0}=\mu_2$ and $\sqrt{\mu_2} R$ is the first zero of $J_1'$, we have $\partial_{\nu\nu}u_m=0$ on $\partial \Omega$. Hence by Lemma \ref{2bel}, $m_0\mu_2=2\pi$.
\end{proof}

The next proposition validates Theorem \ref{exact} in any dimensions.
\begin{proposition}
Let $\Omega$ is a ball of radius $R$ in $\mathbb{R}^n$ and $m_0$ be as in Theorem \ref{gongming}. Then $m_0$ satisfies \eqref{1/2}.
\end{proposition}

\begin{proof}
Let $u_m=r^{1-\frac{n}{2}}J_{\frac{n}{2}-1}(\sqrt{\lambda_m}r)$. Hence when $m>m_0$, such $u_m$ is a solution to \eqref{2problem}. 
First, we claim that on $\partial \Omega$,
\begin{align}
\label{relation}
    \lim_{m \rightarrow m_0^+} \frac{\partial^2 u_m}{\partial r^2}=0
\end{align}To prove this, note that
\begin{align*}
    \frac{\partial u_m}{\partial r}=(1-\frac{n}{2})r^{-\frac{n}{2}}J_{\frac{n}{2}-1}(\sqrt{\lambda_m}r)+\sqrt{\lambda_m} r^{1-\frac{n}{2}}J_{\frac{n}{2}-1}'(\sqrt{\lambda_m}r)=-\sqrt{\lambda_m}r^{1-\frac{n}{2}}J_{\frac{n}{2}}(\sqrt{\lambda_m}r)
\end{align*}and
\begin{align*}
    \frac{\partial^2 u_m}{\partial r^2}=-\sqrt{\lambda_m}(1-\frac{n}{2})r^{-\frac{n}{2}}J_{\frac{n}{2}}(\sqrt{\lambda_m}r)-\lambda_mr^{1-\frac{n}{2}}J_{\frac{n}{2}}'(\sqrt{\lambda_m}r).
\end{align*}
Hence we have \eqref{relation}, since
  \begin{align*}
\left(1-\frac{n}{2}\right)J_{\frac{n}{2}}(\sqrt{\lambda_{m_0}}R)+\sqrt{\lambda_{m_0}}RJ_{\frac{n}{2}}'(\sqrt{\lambda_{m_0}}R)=&\left(1-\frac{n}{2}\right)J_{\frac{n}{2}}(\sqrt{\mu_2(\Omega)}R)+\sqrt{\lambda_{m_0}}RJ_{\frac{n}{2}}'(\sqrt{\mu_2(\Omega)}R)\\
   = &0,
\end{align*}which follows from the fact that $p_{\frac{n}{2},1}:=\sqrt{\mu_2(\Omega)}R$ is the first positive zero of the derivative of the derivative of $t \mapsto t^{1-\frac{n}{2}}J_{\frac{n}{2}}(t)$, see for example \cite[Section 7.4.1]{nshap}.

Now by \eqref{invH} and \eqref{relation}, we have
\begin{align*}
   \left( m_0\mu_2-\frac{n-1}{R}P(\Omega)\right)\int_{\partial \Omega}u\, d\sigma=0.
\end{align*}
Therefore,
\begin{align*}
    m_0\mu_2=\frac{n-1}{R}P(\Omega)=\frac{n-1}{n}\frac{P^2(\Omega)}{|\Omega|}.
\end{align*}
\end{proof}

\section{Comments on solutions when $m \rightarrow 0$}
It seems to be a very challenging task to find out optimal distribution of thermal insulation material when the total amount of material is less than the thresholds $m_1$ and $m_0$ in both problems. Theorem \ref{concentration} and Theorem \ref{q2main} only tell the insulator must vanish on some portion of boundary when $m$ is below the thresholds, but there reveals no information at all on where exactly the thermal insulation material should vanish or concentrate on the boundary.

Nevertheless, for the first problem, as $m\rightarrow 0^+$, the asymptotic distribution of thermal insulation material is known. By the $\Gamma$-convergence framework, it is shown in \cite{EG03} that if $\Omega$ is a $C^2$ domain, then the material should concentrate on the location of $\partial \Omega$ where heat changes fastest. More precisely, the location is given by 
\begin{align}
    \label{location}
\Big\{p \in \partial \Omega: \Big|\frac{\partial u}{\partial \nu}(p)\Big|=\max_{\sigma \in \partial \Omega}\Big|\frac{\partial u}{\partial \nu}(\sigma)\Big| \Big\},    
\end{align}
where $u$ is a solution to \begin{align}
\label{uni0}
\begin{cases}
-\Delta u=1,\quad &\mbox{in $\Omega$}\\
u=0\quad &\mbox{on $\partial \Omega$}.
\end{cases}
\end{align}
The equation \eqref{uni0} describes the stationary temperature inside the domain when the heat source is uniform, the outside temperature is $0$ and there is no insulator. Hence the result in \cite{EG03} really meets the usual intuition. 

Concerning the second problem, let $\lambda_m(\Omega)$ be the infimum in \eqref{2problem}. Since $\lambda_m(\Omega)$ converges to $\lambda_D(\Omega)$, the first eigenvalue of Dirichlet Laplacian, we should consider 
\begin{align}
\label{uni1}
\begin{cases}
-\Delta u=\lambda_D(\Omega) u,\quad &\mbox{in $\Omega$}\\
u=0\quad &\mbox{on $\partial \Omega$}.
\end{cases}
\end{align}
In fact, given an initial positive temperature in $\Omega$ and let the outside temperature be zero, then $\lambda_D(\Omega)$ really gives the lowest decay rate of temperature when there are no heat source inside and no insulator around. One may expect that as $m\rightarrow 0^+$, the optimal insulation also concentrates on the location of boundary described by \eqref{location}, where $u$ is a solution to \eqref{uni1}. It turns out that this is false, because when $\Omega$ is a ball, any solution to \eqref{uni1} is radial, and hence \eqref{location} is the whole boundary, which contradicts to Theorem \ref{gongming}. The question still remains open on where the insulator should asymptotically behave as $m \rightarrow 0$, and it would be interesting to provide an answer.

\section{An open question on Serrin's symmetry results}
As mentioned in the introduction, the two thermal insulation problems not only have concrete physical background, but also produces rich mathematical questions. For example, a quantitative study of concentration breaking in the second problem could possibly be related to new isoperimetric type inequalities. In this section, we propose another open questions, which look very simple but seem to have not been addressed in previous literature.

 Recall that Serrin's overdetermined result classifies $C^2$ critical shapes to torsion energy problem with prescribed volume. More precisely, if $\Omega$ is $C^2$ and there is a solution $u$ such that
\begin{align}
    \label{serrinequation}
\begin{cases}
-\Delta u=1 \quad &\mbox{in $\Omega$}\\
\frac{\partial u}{\partial \nu}=c\quad & \mbox{on $\partial \Omega$}\\
u=0\quad & \mbox{on $\partial \Omega$}
\end{cases}    
\end{align}
Then Serrin proves in \cite{Serrin} that $\Omega$ must be a ball. The proof is by moving plane method, and it validity requires the $C^2$ regularity of domain. Weinberger gave a different proof in \cite{We71} by obtaining integral identities and maximum principle, but still the proof requires the solution is at least in $H^2$ space. 

The following question is motivated in the study of extending Theorem \ref{concentration} to general Lipschitz domains. 

\begin{question}
\label{fundaserrin}
Let $\Omega$ be a bounded Lipschitz domain in $\mathbb{R}^n$ and $u$ be a weak solution to Serrin's overdetermined system \eqref{serrinequation} on $\Omega$. Then is it true that $\Omega$ must be a ball? 
\end{question}
By $u$ being a weak solution to \eqref{serrinequation}, we mean that $u \in H^1_0(\Omega)$ and for any $\phi \in H^1(\Omega)$, we have
\begin{align*}
    \int_{\Omega}\nabla u \nabla \phi \, dx-c\int_{\partial \Omega}\phi \, d\mathscr{H}^{n-1}=\int_{\Omega}\phi \, dx.
\end{align*}
Even if $\Omega$ is $C^1$, regularity results (see for example \cite{FJR78}) can only guarantee that $u \in W^{1,p}_0(\Omega)$ for any $p \in (1,\infty)$, and $H^2$ regularity of solution is missing.

\begin{remark}
Weak solution satisfying $\eqref{serrinequation}_1-\eqref{serrinequation}_2$ always exists on a Lipschitz domain, by Poincar\'e inequality. Hence Question \ref{fundaserrin} really asks whether there is a nonradial domain on which the weak solution satisfying $\eqref{serrinequation}_1-\eqref{serrinequation}_2$ is nontrivial and has trace zero on $\partial \Omega$.
\end{remark}

Extending the symmetry results of Serrin to weak solutions on nonsmooth domains seems to be a very difficult question, and it is not even trivial for $\Omega$ which is $C^2$ everywhere except for a corner or cusp. The latter case is partially dealt with in \cite{Pra98}, where it is shown that if $u$ satisfies \eqref{serrinequation} pointwise except at a corner or cusp on the boundary, and if $\Omega$ satisfies the interior or exterior ball condition, then $\Omega$ must be a ball.



\section*{Acknowledgement}
We would like to thank Professor Carlos Kenig for letting us know his very profound paper \cite{KP93} written with Pipher and suggesting the argument in Remark \ref{Kenigpoint}. We would also like to thank the anonymous referees for giving us many useful suggestions for making the paper much better.

\end{document}